\documentclass[11pt,a4paper]{article}
\usepackage[margin=2.8cm]{geometry}
\usepackage{amsmath,amssymb,amsthm}
\usepackage{microtype}
\usepackage{needspace}
\usepackage[hidelinks]{hyperref}
\hypersetup{pdftitle={Random Permutation Matrices Form a Basis with High Probability},
 pdfauthor={Yijun Jiang},
 pdfsubject={Sharp asymptotics for the spanning probability of random permutation matrices}}
\newtheorem{theorem}{Theorem}[section]
\newtheorem{lemma}[theorem]{Lemma}
\newtheorem{proposition}[theorem]{Proposition}
\newtheorem{corollary}[theorem]{Corollary}
\theoremstyle{remark}
\newtheorem{remark}[theorem]{Remark}
\numberwithin{equation}{section}

\newcommand{\R}{\mathbb{R}}
\newcommand{\F}{\mathbb{F}}
\newcommand{\PP}{\mathbb{P}}
\newcommand{\EE}{\mathbb{E}}
\newcommand{\one}{\mathbf{1}}
\newcommand{\cL}{\mathcal{L}}
\newcommand{\cE}{\mathcal{E}}
\newcommand{\cB}{\mathcal{B}}
\newcommand{\cU}{\mathcal{U}}
\DeclareMathOperator{\per}{per}
\DeclareMathOperator{\rank}{rank}
\DeclareMathOperator{\spanop}{span}

\title{\bfseries Random Permutation Matrices Form a Basis\\with High Probability}
\author{Yijun Jiang\\[6pt]
\small School of Mathematical Sciences, Zhejiang University}
\date{}
\begin{document}
\maketitle

\begin{abstract}
Let $d=(n-1)^2+1$, the dimension of the real linear span of the $n\times n$
permutation matrices. We prove that $d$ independent uniformly random permutation
matrices fail to form a basis with probability
\[
 (1+o(1))n^2(1-1/n)^d=(e^{3/2}+o(1))n^2e^{-n}.
\]
The same asymptotic holds for a uniformly random $d$-element subset, confirming
a conjecture of Kushwaha and Tripathi and identifying the leading obstruction:
a matrix position avoided by every sample. More generally, for any fixed number
of additional samples, we determine the first three exponential orders of the
failure probability. To prove these results, we develop support estimates valid
over arbitrary fields, derive Fourier bounds from permanental minors, and
introduce a counting argument for concentrated assignment functionals over large
prime fields. We also give a sparse lifting argument showing that real rank
deficiency without an annihilating functional of small support has probability
$o(e^{-Cn})$ for every fixed $C>0$.
\end{abstract}

\smallskip
\noindent\textbf{Keywords.} Random permutation matrices; singularity probability;
anti-concentration; permanent; random bases.

\smallskip
\noindent\textbf{2020 Mathematics Subject Classification.} 15B52, 60B20, 05A05.

\section{Introduction}
Let $\mathcal P_n$ be the set of $n\times n$ permutation matrices. For
$\pi\in S_n$, write $(P_\pi)_{ij}=\one_{\{\pi(i)=j\}}$. Their real linear span is
\[
 V_n=\{X\in\R^{n\times n}:\text{all row sums and column sums of $X$
 have one common value}\},
\]
which has dimension $d=(n-1)^2+1$. Kushwaha and Tripathi
\cite[Conjecture~3.4]{KT} asked whether a uniformly random $d$-element subset
of $\mathcal P_n$ is linearly independent with probability tending to one.
Their numerical results suggest an exponentially small failure probability.
We determine its leading asymptotic and the next two exponential orders.

Put $N=n^2$. Let $\pi_1,\ldots,\pi_m$ be independent uniform permutations,
where $m=d+s$ and $s\ge0$ is a fixed integer, and define
\[
 \mathcal F_{n,s}=\{\spanop_{\R}(P_{\pi_1},\ldots,P_{\pi_m})\ne V_n\}.
\]
Unless otherwise stated, probabilities refer to this independent sampling model.

\begin{theorem}[Sharp asymptotic]\label{thm:main}
For every fixed integer $s\ge0$, as $n\to\infty$,
\begin{equation}\label{eq:main}
 \PP(\mathcal F_{n,s})\sim n^2(1-1/n)^{d+s}\sim e^{3/2}n^2e^{-n}.
\end{equation}
The same statement holds for a uniformly random $(d+s)$-element subset of
$\mathcal P_n$. In particular, when $s=0$, the sampled matrices form a basis
of $V_n$ with the complementary probability.
\end{theorem}

A fixed matrix position is avoided by all $m$ samples with probability
$(1-1/n)^m$. Its coordinate functional is then a nonzero functional on $V_n$
annihilating all the samples. Theorem~\ref{thm:main} shows that the union of
these $n^2$ events accounts for the entire leading term.

An assignment functional has the form $b+\sum_i a_{i,\pi(i)}$. Adding constants
to rows or columns changes the assignment sum by a constant, which can be
absorbed into $b$. Its \emph{support complexity} is the least number of nonzero
entries in an array representing it after such additions. The precise definition
is given in Section~\ref{sec:coordinates}. Fix a sufficiently small absolute
constant $\eta>0$ and set $t=\lfloor\eta d\rfloor$. Let $\cL_{m,\R}$ denote
the event that the samples admit a nonzero annihilating assignment functional
of complexity less than $t$.

\begin{theorem}[Real remainder]\label{thm:remainder}
For every fixed integer $s\ge0$ and every fixed $C>0$,
\begin{equation}\label{eq:remainder}
 \PP(\mathcal F_{n,s}\setminus\cL_{d+s,\R})=o(e^{-Cn}).
\end{equation}
This also holds for uniform sampling without replacement.
\end{theorem}

Theorem~\ref{thm:remainder} reduces the next exponential orders to relations
supported on a few matrix positions. Put $m=d+s$, $u=1/[n(n-1)]$, and define
\begin{equation}\label{eq:T}
\begin{split}
 T_{n,s}={}&n^2(1-1/n)^m-n^2(n-1)(1-2/n)^m\\
 &+\frac{n^2(n-1)^2}{2}
 \bigl[(1-2/n+2u)^m-2(1-2/n+u)^m\bigr].
\end{split}
\end{equation}

\begin{theorem}[Higher exponential orders]\label{thm:higher}
For every fixed integer $s\ge0$, with $m=d+s$,
\begin{align}
 \PP(\mathcal F_{n,s})
 &=n^2(1-1/n)^m+
 \left(\frac{e^4-2e^3}{2}+O_s(n^{-1})\right)n^4e^{-2n},\label{eq:second}\\
 \PP(\mathcal F_{n,s})
 &=T_{n,s}+\left(\frac{e^{9/2}}{6}(3e^2-9e+7)+O_s(n^{-1})\right)
 n^6e^{-3n}.\label{eq:third}
\end{align}
More precisely, the explicit expression $S_{n,s}$ defined in \eqref{eq:S} satisfies
\begin{equation}\label{eq:exacthigher}
 \PP(\mathcal F_{n,s})=S_{n,s}+O_s(n^8e^{-4n}).
\end{equation}
All three statements remain valid for uniform subsets.
\end{theorem}

The two-position expression $T_{n,s}$ is retained exactly in
\eqref{eq:third}: a finite expansion in powers of $1/n$ would leave errors larger
than the third exponential order.

The first version of this paper\footnote{\href{https://arxiv.org/abs/2609.05581v1}{arXiv:2609.05581v1}.}
proved the bound $O(n^{-1/2})$ and recorded the empty-position lower bound.
Its proof combined binary Fourier estimates with the kernel decomposition
of Ferber, Kwan, and Sauermann~\cite{FKS}, using a characteristic-function
bound of Roos in the form recorded by Do, Nguyen, Phan, Tran, and
Vu~\cite{DNP}. The present version substantially strengthens those results.
The new arguments developed here by the author identify the empty-position
obstruction as the complete leading-order contribution, establish the sharp
asymptotic, and determine two further exponential orders. They are based on a
different proof architecture, combining support estimates over arbitrary
fields, Fourier bounds from permanental minors, a counting argument for
concentrated assignment functionals over large prime fields, and a sparse
lifting argument from modular to real rank deficiency. In particular, the
proof of the sharp asymptotic does not proceed by refining the $O(n^{-1/2})$
argument from the first version.

The principal difficulty is to rule out annihilating functionals of large support
whose assignment sums have substantial atoms. We derive Fourier bounds using
blocks of any fixed size from Roos's inequalities for averaged squared permanental
minors~\cite{Roos}. At small Fourier levels, an aggregate sampling argument describes
almost all rectangle differences using a small set of values. At large levels,
a positive proportion of matching differences yields many independent linear
constraints. Together these estimates control concentrated normals over a large
prime field at every fixed exponential scale. A separate lifting argument
transfers the required estimate to real normals without identifying sparse
relations over different fields.
The proof is self-contained apart from the stated permanental inequality,
the Cauchy--Davenport inequality, and standard probabilistic and
linear-algebraic estimates.

Sections~\ref{sec:coordinates}--\ref{sec:sparse} establish the linear algebra
and support estimates. Sections~\ref{sec:fourier}--\ref{sec:counting} prove the
finite-field concentration bound. Section~\ref{sec:real} proves
Theorems~\ref{thm:main} and~\ref{thm:remainder}, and
Section~\ref{sec:higher} computes the higher exponential orders.
All logarithms are natural. Implicit constants are absolute unless their
dependence on a parameter is indicated.

\section{Coordinates and subspace compression}\label{sec:coordinates}
All linear algebra in this section is over an arbitrary field $\mathbb K$.
For a permutation $\pi$, define
\begin{equation}\label{eq:Y}
 Y_\pi=\bigl(1,(\one_{\{\pi(i)=j\}})_{i,j<n}\bigr)\in\mathbb K^d.
\end{equation}
A matrix with common row and column sum $r$ is uniquely determined by
$r$ and its upper-left $(n-1)\times(n-1)$ block, since
\[
 x_{in}=r-\sum_{j<n}x_{ij},\quad
 x_{nj}=r-\sum_{i<n}x_{ij},\quad
 x_{nn}=(2-n)r+\sum_{i,j<n}x_{ij}.
\]
Thus the coordinate map is an isomorphism onto $\mathbb K^d$.
Moreover, the vectors $Y_\pi$ span $\mathbb K^d$. For each $i,j<n$, two
permutations can be chosen that differ only by replacing the edges $(i,j),(n,n)$
with $(i,n),(n,j)$. The difference of their coordinate vectors is $(0,e_{ij})$.
Any permutation vector supplies the remaining constant coordinate. This also
proves the assertion about $V_n$ and its dimension.

Let $M_m$ be the $m\times d$ matrix with rows $Y_{\pi_r}$. A \emph{normal}
is a nonzero vector $x\in\mathbb K^d$, regarded as the functional
$\pi\mapsto x\cdot Y_\pi$. Write
\[
 f_{b,A}(\pi)=b+\sum_{i=1}^n a_{i,\pi(i)},\qquad
 A\in\mathbb K^{n\times n}.
\]
Two pairs represent the same functional precisely when
\begin{equation}\label{eq:equivalence}
 a'_{ij}=a_{ij}+r_i+c_j,\qquad
 b'=b-\sum_i r_i-\sum_j c_j
\end{equation}
for some $r_i,c_j\in\mathbb K$. For the converse, compare permutations
differing on a rectangle: an array with constant assignment sum has every
rectangle difference zero, and is therefore a sum of a row term and a column
term. Each pair has a unique representative whose array has zero last row and
last column; its remaining coefficients and constant are the coordinates of $x$.
Define
\begin{equation}\label{eq:complexity}
\begin{split}
 \kappa_{\mathbb K}(x)&=\min_{r_i,c_j\in\mathbb K}
       \#\{(i,j):a_{ij}+r_i+c_j\ne0\},\\
 \rho_{\mathbb K}(x)&=\max_{z\in\mathbb K}\PP(x\cdot Y_\pi=z).
\end{split}
\end{equation}
The maximum exists because $Y_\pi$ has finite support. These quantities do not
depend on the chosen representation. We also write $\kappa_{\mathbb K}(A)$
and $\rho_{\mathbb K}(A)$ when the affine constant is immaterial. Set
\begin{equation}\label{eq:L}
 \cL_{m,\mathbb K}=
 \{\exists x\in\mathbb K^d\setminus\{0\}:M_mx=0,
             \ \kappa_{\mathbb K}(x)<t\}.
\end{equation}

The following estimate counts families of normals without enumerating their
coefficients.
\begin{lemma}[Subspace compression]\label{lem:compression}
Let $Z_1,\ldots,Z_m$ be independent copies of a finitely supported random
vector in $\mathbb K^D$, let $U\subseteq\mathbb K^D$ be a fixed subspace
of dimension $q\ge1$, and let $0<v\le1$. The probability that some nonzero
$x\in U$ annihilates every $Z_i$ and satisfies $\PP(x\cdot Z_1=0)\le v$
is at most
\begin{equation}\label{eq:compression}
 \sum_{j=0}^{\min(q-1,m)}\binom mj v^{m-j}.
\end{equation}
If $q\le m$, this is at most $2^m v^{m-q+1}$.
\end{lemma}
\begin{proof}
On the event in question, the restrictions to $U$ of the sample functionals
$Z_i\cdot(\,\cdot\,)$ have rank $j\le q-1$. Choose a set $I\subseteq[m]$
of size $j$ spanning these restrictions, and put
\[
 K_I=\{x\in U:x\cdot Z_i=0\text{ for every }i\in I\}.
\]
Condition on the samples indexed by $I$. Every remaining sample must
annihilate all of $K_I$. If $K_I$ contains an eligible witness $x$, then
\[
 \PP(Z_1\cdot y=0\text{ for all }y\in K_I)
 \le\PP(Z_1\cdot x=0)\le v.
\]
If it contains no such witness, the event under consideration is impossible.
Independence bounds the conditional probability by $v^{m-j}$. Summing over
$I$ and $j$ proves \eqref{eq:compression}. When $q\le m$, all exponents
are at least $m-q+1$.
\end{proof}

\begin{corollary}[Rank-sensitive compression]\label{cor:rankcompression}
In the setting of Lemma~\ref{lem:compression}, the probability that the sample
kernel in $\mathbb K^D$ has dimension at least $r\ge1$ and contains a
nonzero vector with zero mass at most $v$ is at most
\begin{equation}\label{eq:rankcompression}
 \sum_{j=0}^{\min(D-r,m)}\binom mj v^{m-j}.
\end{equation}
If $h=m-D+r\ge1$ and $mv\le1/2$, this is at most $2\binom mh v^h$.
\end{corollary}
\begin{proof}
In the preceding proof, take $U=\mathbb K^D$ and restrict to ranks
$j\le D-r$. Reindex by $\ell=m-j\ge h$. The ratio of consecutive terms
$\binom m\ell v^\ell$ is at most $mv\le1/2$.
\end{proof}
An affine family of normals of dimension at most $q$ is contained in a linear
subspace of dimension at most $q+1$. We shall apply
Lemma~\ref{lem:compression} to this span. Only the witness must satisfy the
concentration bound; the other vectors in the span need not do so.

\section{Normals of small support}\label{sec:sparse}
For an array $A$, let
\[
 D_0(A)=\#\{(i,j,a,b)\in[n]^4:a_{ia}-a_{ja}-a_{ib}+a_{jb}\ne0\}.
\]
For each anchor $(j,b)$, the array of differences indexed by $(i,a)$ is
row-column equivalent to $A$. Therefore
\begin{equation}\label{eq:rectangles}
 D_0(A)\ge n^2\kappa_{\mathbb K}(A).
\end{equation}

\begin{lemma}[Uniform atom bounds]\label{lem:atoms}
For every field $\mathbb K$ and all sufficiently large $n$, the following hold.
If $1\le k=\kappa_{\mathbb K}(x)\le n/4$, then
\begin{equation}\label{eq:rk}
 \PP(x\cdot Y_\pi=0)\le r_k,
 \qquad r_k=1-\frac{k(n-k)}{n(n-1)}.
\end{equation}
If $\kappa_{\mathbb K}(x)\ge n/4$, then $\rho_{\mathbb K}(x)\le21/22$.
\end{lemma}
\begin{proof}
Choose a representative with $k$ nonzero entries, and let $Z$ count the support
positions occupied by $\pi$. Then
\[
 \EE Z=\frac{k}{n},\qquad \EE Z(Z-1)\le\frac{k(k-1)}{n(n-1)}.
\]
Since $\one_{\{Z=1\}}\ge Z-Z(Z-1)$ for nonnegative integer $Z$,
\begin{equation}\label{eq:onehit}
 \PP(Z=1)\ge\frac{k(n-k)}{n(n-1)}.
\end{equation}
If the affine constant is zero, one occupied support position prevents
vanishing. If it is nonzero, vanishing requires $Z\ge1$, an event of
probability at most $k/n\le r_k$. This proves \eqref{eq:rk}.

For the second assertion, let $g_e$ be the number of unordered column pairs
producing a nonzero rectangle difference for an unordered row pair $e$.
Then $\sum_e g_e=D_0(A)/4$. A uniformly random matching of the rows,
leaving one row unmatched when $n$ is odd, contains each row pair with
probability at least $1/n$. Hence some row matching $\mathcal M$ satisfies
$\sum_{e\in\mathcal M}g_e\ge D_0(A)/(4n)$.
Let $G$ count its edges whose image column pairs give nonzero differences.
Then
\[
 \mu:=\EE G\ge\frac{D_0(A)}{2n^2(n-1)}
 \ge\frac{\kappa_{\mathbb K}(A)}{2n}\ge\frac18.
\]
For two disjoint row pairs, the joint probability of good image pairs is at
most three times the product of their individual probabilities, since
$n(n-1)/[(n-2)(n-3)]\le3$ for $n\ge6$. Thus
\[
 \EE G^2\le\mu+3\mu^2,\qquad
 \PP(G>0)\ge\frac{\mu^2}{\EE G^2}\ge\frac1{11}.
\]
Condition on the unordered images of the matched pairs and on the unmatched
image, if any. The pair orientations are independent fair choices. If $G>0$,
conditioning further on all but one good pair leaves two distinct assignment
sums with equal probability. Every conditional atom is therefore at most
$1/2$. Averaging proves $\rho_{\mathbb K}(x)\le1-1/22$.
\end{proof}

Fix $0<\eta<1/4$ so small that
\begin{equation}\label{eq:eta}
 2H(\eta)+(1-\eta)\log(21/22)<0,
 \qquad H(a)=-a\log a-(1-a)\log(1-a),
\end{equation}
and put $t=\lfloor\eta d\rfloor$. Let $\cE_m$ be the event that some
matrix position is avoided by all $m$ samples.

\begin{proposition}[Sparse probability estimates]\label{prop:sparse}
Uniformly over fields $\mathbb K$ and $m\ge d-1$,
\begin{equation}\label{eq:sparse}
 \PP(\cL_{m,\mathbb K}\setminus\cE_m)=O(n^8e^{-2n}).
\end{equation}
Moreover, uniformly over fields and $m\ge d$, the probability of a vanishing
normal with $4\le\kappa_{\mathbb K}(x)<t$ is $O(n^8e^{-4n})$.
\end{proposition}
\begin{proof}
For a fixed support of size $k$, the corresponding normals, including the affine
constant, lie in a subspace of dimension at most $k+1$. For witnesses of
minimum support $k\le n/4$, Lemmas~\ref{lem:compression} and~\ref{lem:atoms}
give the bound
\begin{equation}\label{eq:supportcount}
 \binom Nk\sum_{j=0}^k\binom mj r_k^{m-j}.
\end{equation}
For $n/4\le k<t$, the same bound holds with $21/22$ in place of $r_k$.
First take $m\in\{d-1,d\}$.

A nonzero normal of complexity zero is a nonzero constant, so cannot vanish.
A complexity-one relation makes one occupancy indicator constant across all
samples. The position is either always empty or always occupied; the latter
events have total probability at most $Nn^{-m}$. For $k=2$,
\eqref{eq:supportcount} is $O(n^8e^{-2n})$. For $3\le k\le n/4$, put $m_0=d-1$. Since $m\ge m_0$ and
$-\log(1-x)\ge x$ for $0\le x<1$,
\[
 -(m-k)\log r_k
 \ge(m_0-k)\frac{k(n-k)}{n(n-1)}
 \ge\left(\frac34-O(n^{-1})\right)kn.
\]
On the other hand,
\[
 \log\left(\binom Nk\sum_{j=0}^k\binom mj\right)
 \le k\log(eN/k)+(k+1)\log(m+1)=O(k\log n),
\]
uniformly for $m\in\{d-1,d\}$. Thus \eqref{eq:supportcount} is at most
$e^{-7kn/10}$ for large $n$, uniformly in this range. Its sum is
$o(e^{-2n})$.
For the remaining supports, the binomial entropy bound and \eqref{eq:eta} give
\begin{equation}\label{eq:entropytail}
\begin{split}
 \sum_{n/4\le k<t}\binom Nk\sum_{j=0}^k\binom mj(21/22)^{m-j}
 &\le\exp\bigl(\{2H(\eta)+(1-\eta)\log(21/22)+o(1)\}N\bigr)\\
 &\le e^{-cN}
\end{split}
\end{equation}
for some $c>0$. This proves \eqref{eq:sparse} for $m=d-1,d$.
For larger $m$, the estimates for complexity at least two use only the first
$d$ samples, while complexity-one relations are classified using all $m$
samples. This proves the asserted uniform bound.

For the second assertion, a support of size $k$ has $2^k$ binary occupancy
patterns and therefore admits at most $2^{2^k}$ possible zero sets of a
functional, regardless of its coefficients or the field. For fixed
$k\in\{4,5,6\}$, enumerate only zero sets arising from normals of minimum
support $k$. Each has one-sample probability at most $r_k$, so the resulting
bound, using the first $d$ samples, is
\[
 O_k\left(\binom Nk r_k^d\right)=O_k(n^{2k}e^{-kn}).
\]
For later use, note also that, uniformly for $7\le k\le n/4$ and
$m_0=d-1$, the same logarithmic estimates give
\begin{equation}\label{eq:mediumtail}
 \binom Nk(m_0+1)^{k+1}r_k^{m_0-k}\le e^{-3kn/5}.
\end{equation}
The sum over $k\ge7$ is $O(e^{-21n/5})$, while \eqref{eq:entropytail}
controls larger supports. These errors are $O(n^8e^{-4n})$.
\end{proof}

\section{Fourier estimates and finite descriptions}\label{sec:fourier}
Let $Q$ be prime and let the coefficients belong to $\F_Q$.
For $\xi,z\in\F_Q$, write $\chi_\xi(z)=\exp(2\pi i\xi z/Q)$,
using residue representatives in the exponent. Set
\[
 \varphi_A(\xi)=\EE_\pi\chi_\xi\left(\sum_i a_{i,\pi(i)}\right),
 \qquad Z_\xi=(\chi_\xi(a_{ij})).
\]
For $1\le k\le n$, define
\begin{equation}\label{eq:Fk}
 F_k(A;\xi)=\binom nk^{-2}
 \sum_{\substack{R,C\subseteq[n]\\|R|=|C|=k}}
 \left|\frac{\per(Z_\xi[R,C])}{k!}\right|^2.
\end{equation}
These quantities lie in $[0,1]$, satisfy $F_1=1$, and are invariant under
row-column additions to $A$.

\begin{proposition}[Block Fourier bounds]\label{prop:fourier}
For positive integers $j,\ell$ with $j+\ell\le n$, one has
$F_{j+\ell}\le F_jF_\ell$. If $2\le k\le n/2$ and
$p=\lfloor n/(2k)\rfloor$, then
\begin{equation}\label{eq:fourierbound}
 |\varphi_A(\xi)|\le F_k(A;\xi)^p,
 \qquad \rho_Q(A)\le\frac1Q\sum_{\xi\in\F_Q}F_k(A;\xi)^p.
\end{equation}
Here and below, the subscript $Q$ denotes the field $\F_Q$.
If $R,C$ are independent uniform $k$-subsets and $\sigma,\tau:R\to C$
are independent uniform bijections, then
\begin{equation}\label{eq:matchingdifference}
 F_k(A;\xi)=\EE\chi_\xi(\Delta_A),
 \qquad \Delta_A=\sum_{i\in R}(a_{i,\sigma(i)}-a_{i,\tau(i)}).
\end{equation}
The probability of equal matchings is $q_k=1/k!$.
\end{proposition}
\begin{proof}
Roos's Theorem~1.3, equations~(20) and~(22) in the arXiv version
of~\cite{Roos}, gives submultiplicativity and bounds the normalized full
permanent by the product of square roots of the averaged squared minors for
any prescribed partition of $n$. Use $2p$ parts of size $k$ and the remaining
parts of size one. Since $F_1=1$, this proves the first bound in
\eqref{eq:fourierbound}. Fourier inversion and the triangle inequality give
the second. Expanding the squared permanent proves
\eqref{eq:matchingdifference}; conditioning on one bijection gives
$\PP(\sigma=\tau)=1/k!$.
\end{proof}
In particular, $F_k\le F_2^{\lfloor k/2\rfloor}$. The exact pair formula is
\begin{equation}\label{eq:F2}
 F_2(A;\xi)=\frac1{n^2(n-1)^2}
 \sum_{\substack{i\ne j\\a\ne b}}
 \cos^2\left(\frac{\pi\xi}{Q}(a_{ia}-a_{ja}-a_{ib}+a_{jb})\right).
\end{equation}
The equal-matching contribution sets the scale
\begin{equation}\label{eq:ck}
 c_k=\frac{\log(k!)}{2k},\qquad c_k\longrightarrow\infty,
 \qquad q_k^p=\exp(-c_kn+O_k(1)).
\end{equation}
For a prescribed exponential concentration threshold $e^{-Bn}$, we shall
choose a fixed $k$ with $c_k>B$.

For a nonempty set $L\subseteq\F_Q$, put
\[
 \mu_L(z)=\frac1{|L|}\sum_{\xi\in L}\|\xi z/Q\|^2,
 \qquad X_a(L)=\{z\in\F_Q:\mu_L(z)\le a\},
\]
where $\|\,\cdot\,\|$ denotes distance to the nearest integer.

\begin{lemma}[Small-energy cardinality]\label{lem:energy}
If $|L|>4$ and $0<a\le1/1600$, then
\begin{equation}\label{eq:energycardinality}
 |X_a(L)\setminus\{0\}|\le80\sqrt a\,\frac Q{|L|}.
\end{equation}
\end{lemma}
\begin{proof}
Set $j=\lfloor(10\sqrt a)^{-1}\rfloor$, so
$j\ge(20\sqrt a)^{-1}$. The triangle inequality on the circle and
Minkowski's inequality imply $jX_a\subseteq X_{j^2a}\subseteq X_{1/100}$,
where $jX_a$ is a sumset. At each point of this sumset,
\[
 \Re\sum_{\xi\in L}\chi_\xi(z)
 \ge(1-2\pi^2/100)|L|>|L|/2.
\]
Orthogonality gives $|jX_a|\le4Q/|L|<Q$. The prime-field
Cauchy--Davenport inequality, in its iterated form
$|jX|\ge\min\{Q,j(|X|-1)+1\}$; see~\cite{Wheeler}, now yields
$j(|X_a|-1)+1\le4Q/|L|$. Since $0\in X_a$, this proves the claim.
\end{proof}

\begin{lemma}[Aggregate sampling]\label{lem:aggregate}
Fix $0<\varepsilon<1/6400$ and $0<\zeta<1$. For all sufficiently large
primes $Q$, uniformly over $|L|>4$ and lists $z_1,\ldots,z_J\in\F_Q$
satisfying
\begin{equation}\label{eq:aggregatehyp}
 \frac1J\sum_{a=1}^J\mu_L(z_a)\le\frac{\zeta\varepsilon}{4},
\end{equation}
there is a tuple $\xi_1,\ldots,\xi_\ell\in L$, with
$\ell=\lceil6400\log(2Q)\rceil$, such that
\begin{equation}\label{eq:aggregateB}
 \cB=\left\{z\in\F_Q:\frac1\ell\sum_{j=1}^\ell
           \|\xi_jz/Q\|^2\le2\varepsilon\right\}
\end{equation}
satisfies
\begin{equation}\label{eq:aggregateconclusion}
 \#\{a:z_a\notin\cB\}\le\zeta J,
 \qquad |\cB\setminus\{0\}|\le650\sqrt\varepsilon\,\frac Q{|L|}.
\end{equation}
The coefficient in $\ell=O(\log Q)$ is independent of $\varepsilon$;
the lower bound on $Q$ may depend on $\varepsilon$.
\end{lemma}
\begin{proof}
Sample the frequencies independently and uniformly from $L$. If
$\mu_L(z)\ge4\varepsilon$, the multiplicative Chernoff bound for
$4\|\xi_jz/Q\|^2\in[0,1]$ gives
\[
 \PP(z\in\cB)\le\exp(-\ell\mu_L(z)/2).
\]
Split nonzero $z$ into the ranges $\mu_L(z)\le4\varepsilon$,
$4\varepsilon<\mu_L(z)\le1/1600$, and $\mu_L(z)>1/1600$.
By Lemma~\ref{lem:energy}, their contributions to
$\EE|\cB\setminus\{0\}|$ are bounded respectively by
\[
 160\sqrt\varepsilon\,\frac Q{|L|},\qquad
 2\frac Q{|L|}e^{-2\ell\varepsilon},\qquad
 Qe^{-\ell/3200}\le\frac1{4Q}.
\]
For fixed $\varepsilon$ and large $Q$, each of the last two terms is at
most $\sqrt\varepsilon Q/|L|$. Markov's inequality bounds the probability
of violating the cardinality assertion by $162/650<1/4$.

For each listed value, ordinary Markov gives
$\PP(z_a\notin\cB)\le\mu_L(z_a)/(2\varepsilon)$. The expected number
of excluded entries is at most $\zeta J/8$, so the coverage assertion fails
with probability at most $1/8$. Both assertions therefore hold for some
tuple. This argument gives average coverage of the specified list; it does
not require uniform inclusion of an entire energy level set.
\end{proof}

\begin{lemma}[Positive Fourier bias]\label{lem:bias}
Fix $k\ge2$ and $q_k<\theta<1$, and put
$\delta=(\theta-q_k)/(1-q_k)$. Suppose $L\ne\varnothing$ and
$F_k(A;\xi)\ge\theta$ for every $\xi\in L$. There is a set
$\cB\subseteq\F_Q$, described by a frequency tuple of length
\begin{equation}\label{eq:biaslength}
 \ell=\lceil128\delta^{-2}\log(4Q)\rceil,
\end{equation}
such that $|\cB|\le16Q/(\delta^2|L|)$ and at least a $\delta/2$
fraction of the unequal-matching tuples in \eqref{eq:matchingdifference}
satisfy $\Delta_A\in\cB$. There are at most $Q^\ell$ descriptions.
\end{lemma}
\begin{proof}
Write $C_L(z)=|L|^{-1}\sum_{\xi\in L}\cos(2\pi\xi z/Q)$.
Taking real parts in \eqref{eq:matchingdifference}, averaging over $L$, and
removing the equal-matching contribution gives
$\EE(C_L(\Delta_A)\mid\sigma\ne\tau)\ge\delta$.
Since $C_L\le1$, at least a $\delta/(2-\delta)\ge\delta/2$ fraction of
these tuples satisfy $C_L(\Delta_A)\ge\delta/2$.

Sample $\ell$ independent uniform frequencies from $L$. For each $z$,
Hoeffding's inequality bounds the probability that the sample mean cosine
differs from $C_L(z)$ by at least $\delta/8$ by
$2\exp(-\ell\delta^2/128)$. The union bound over all $Q$ values is at
most $1/2$. Choose a tuple for which all deviations are smaller, and let
\[
 \cB=\left\{z:\frac1\ell\sum_{j=1}^\ell
                 \cos(2\pi\xi_jz/Q)\ge3\delta/8\right\}.
\]
This set contains $\{z:C_L(z)\ge\delta/2\}$ and is contained in
$\{z:C_L(z)\ge\delta/4\}$. Orthogonality yields
\[
 \sum_{z\in\F_Q}\left|\sum_{\xi\in L}\chi_\xi(z)\right|^2=Q|L|,
\]
which gives $|\cB|\le16Q/(\delta^2|L|)$. The tuple describes $\cB$
without separately recording $L$.
\end{proof}

\Needspace{15\baselineskip}
\section{Counting concentrated normals}\label{sec:counting}
Throughout this section, $m_0=d-1=N-2n+1$.

\begin{theorem}[Concentrated normals over a large prime field]\label{thm:finite}
For every fixed $B>0$, there are $c_B>0$ and $n_B$ such that, if $n\ge n_B$
and $Q$ is prime with
\begin{equation}\label{eq:prime}
 e^{2Bn}\le Q\le2e^{2Bn},
\end{equation}
then
\begin{equation}\label{eq:finite}
 \PP\left(\begin{array}{c}
 \exists x\in\F_Q^d\setminus\{0\}:M_{m_0}x=0,\\
 \kappa_Q(x)\ge t,\ \rho_Q(x)\ge e^{-Bn}
 \end{array}\right)\le e^{-c_Bn^2}.
\end{equation}
\end{theorem}

Let $V=V(R,C,\sigma,\tau)$ be the signed difference matrix of the two
matchings in \eqref{eq:matchingdifference}. Thus $\Delta_A=\langle V,A\rangle$,
and every row and column sum of $V$ is zero. Before identifying repeated
difference matrices, the number of equally likely tuples is
\begin{equation}\label{eq:tuplecount}
 T_k=\binom nk^2(k!)^2=(n)_k^2\le n^{2k},
\end{equation}
where $(n)_k=n(n-1)\cdots(n-k+1)$.

\begin{lemma}[Rank of matching differences]\label{lem:matchingrank}
For every $r$-dimensional subspace $S\subseteq\F_Q^{n\times n}$,
\begin{equation}\label{eq:matchingrank}
 \PP(V\in S\setminus\{0\})\le\frac{2kr}{N}.
\end{equation}
Consequently, any family of tuples of probability at least $g>0$, all with
$V\ne0$, contains $\lfloor gN/(4k)\rfloor$ linearly independent difference
matrices for all sufficiently large $n$.
\end{lemma}
\begin{proof}
There are $r$ matrix positions whose coordinate restriction is injective on
$S$: take pivot columns of a matrix with a basis of $S$ as its rows.
Every nonzero element of $S$ is nonzero in one of these positions.
A fixed position belongs to either matching with probability at most $2k/N$.
The union bound proves \eqref{eq:matchingrank}. During a greedy selection
of fewer than $gN/(4k)$ independent differences, the nonzero probability
mass in their span is less than $g/2$, so a tuple from the family remains
outside that span. Unequal matchings have nonzero difference over every
field, because each nonzero entry is $1$ or $-1$.
\end{proof}

\subsection{Selecting a Fourier level}
\begin{proof}[Proof of Theorem~\ref{thm:finite}]
Fix $B>0$. Choose a fixed integer $k\ge2$ with $c_k>B$, then fix
\begin{equation}\label{eq:parameters}
 q_k<\theta<e^{-2kB},\qquad
 \delta=\frac{\theta-q_k}{1-q_k},\qquad
 g=\frac{(1-q_k)\delta}{2},\qquad
 r=\left\lfloor\frac{gN}{4k}\right\rfloor.
\end{equation}
Put $p=\lfloor n/(2k)\rfloor$. For large $n$, $r\ge2$.
For an eligible witness with array representative $A$, write
$\rho=\rho_Q(A)\ge e^{-Bn}$ and set
\[
 W(\xi)=-p\log F_k(A;\xi),\qquad H=\lceil\log(2/\rho)\rceil,
\]
with $W=+\infty$ where $F_k=0$. By \eqref{eq:fourierbound},
$\rho\le Q^{-1}\sum_\xi e^{-W(\xi)}$. If every interval $[j,j+1)$,
$0\le j<H$, contained fewer than $(\rho Q/8)e^{j/2}$ values of $W$, then
\[
 \frac1Q\sum_\xi e^{-W(\xi)}
 <\frac\rho8\sum_{j\ge0}e^{-j/2}+e^{-H}<\rho,
\]
a contradiction. Hence for some integer $1\le h\le H$,
\begin{equation}\label{eq:level}
 L=\{\xi:W(\xi)\le h\},\qquad
 |L|\ge\frac{\rho Q}{8}e^{(h-1)/2}.
\end{equation}
We have $H\le Bn+O(1)$ and $|L|>4$ for large $n$. Uniformly over these
levels,
\begin{equation}\label{eq:theta}
 F_k(A;\xi)\ge e^{-h/p}\ge\theta\qquad(\xi\in L),
\end{equation}
by the strict inequality in \eqref{eq:parameters}.

Cover $[e^{-Bn},1]$ by the dyadic intervals $[v/2,v]$ with
$v=2^{-j}$, $0\le j\le\lfloor Bn/\log2\rfloor$. There are $O_B(n)$
such intervals, and every upper endpoint satisfies $v\ge e^{-Bn}$. Within one interval the witness has zero mass at most $v$.
There are also $O_B(n)$ choices of $h$. We count descriptions separately
in the two ranges divided by
\begin{equation}\label{eq:hstar}
 h_*=(4k+4)\log n.
\end{equation}

For each fixed $v$ and $h$, all subspaces used in the union bounds below
are determined by frequency tuples, matrix positions, and field values.
These descriptions range over deterministic finite sets, independently of
the sampled permutations. The witness selects one of the descriptions,
but does not change the family over which the union bound is taken.

\subsection*{Small levels: $h\le h_*$}
Let $a_k=\lfloor k/2\rfloor$. Since $F_k\le F_2^{a_k}$, on $L$ we have
\[
 1-F_2(A;\xi)\le1-e^{-h/(pa_k)}\le\frac{h}{pa_k}.
\]
For large $n$, $pa_k\ge n/12$. Formula~\eqref{eq:F2} and
$\sin^2(\pi x)\ge4\|x\|^2$ give
\begin{equation}\label{eq:rectangleenergy}
 \frac1{n^4}\sum_{i,j,a,b}
 \mu_L(a_{ia}-a_{ja}-a_{ib}+a_{jb})\le\frac{3h}{n}.
\end{equation}
Repeated row or column indices contribute zero.

Set $\zeta=\eta/8$, and choose a fixed $\varepsilon>0$, to be specified
below. Since $h\le h_*=O_k(\log n)$, the right-hand side of
\eqref{eq:rectangleenergy} is at most $\zeta\varepsilon/4$ for sufficiently
large $n$. Lemma~\ref{lem:aggregate} supplies a set $\cB$ containing all
but $\zeta n^4$ of the actual rectangle differences. By
\eqref{eq:rectangles}, at least $n^2t>\zeta n^4$ differences are nonzero,
so $\cB$ contains a nonzero element. The set in \eqref{eq:aggregateB} is
symmetric and contains zero. Since $Q$ is odd, it follows that
$|\cB|\le(3/2)|\cB\setminus\{0\}|$. Equations~\eqref{eq:aggregateconclusion}
and~\eqref{eq:level} imply, for an absolute constant $C_0$,
\begin{equation}\label{eq:smallB}
 |\cB|\le C_0\sqrt\varepsilon\,\rho^{-1}e^{-h/2}
 \le\frac\gamma v,\qquad \gamma=2C_0\sqrt\varepsilon.
\end{equation}
There are at most $Q^{\ell_0}$ possible descriptions of $\cB$, where
$\ell_0=\lceil6400\log(2Q)\rceil$ does not depend on $\varepsilon$.

Average the exceptional rectangle differences over the $n^2$ anchors.
For some anchor the resulting representative of $A$ has at most
$a\le\zeta N$ entries outside $\cB$. Fix these positions and the values
of the other $N-a$ entries in $\cB$, leaving the exceptional entries and
the affine constant free. Their image in canonical coordinates is an affine
family of dimension at most $a+1$, contained in a linear subspace of
dimension at most $a+2$. Every witness belongs to such a family. There is
no need to record the anchor: all the resulting arrays are already included
in this enumeration.

Enumerate only frequency descriptions satisfying \eqref{eq:smallB}.
Lemma~\ref{lem:compression} bounds the contribution of this level and interval by
\begin{equation}\label{eq:smallcount}
\begin{split}
 Q^{\ell_0}\sum_{a\le\zeta N}\binom Na
       (\gamma/v)^{N-a}2^{m_0}v^{m_0-a-1}
 &\le Q^{\ell_0}2^{N+m_0}\gamma^{(1-\zeta)N}v^{-2n}.
\end{split}
\end{equation}
Here $m_0-N-1=-2n$, and we require $\gamma<1$.
For large $n$, $\ell_0\log Q\le6500(\log Q)^2$ and $\log Q\le3Bn$.
Thus the logarithm of \eqref{eq:smallcount} is at most
\begin{equation}\label{eq:smalllog}
 \bigl[60000B^2+2\log2+2B+(1-\zeta)\log\gamma\bigr]N.
\end{equation}
Choose $0<\varepsilon<1/6400$ so small that the bracket is at most $-10$.
This is possible because the frequency description length has no factor
depending on $1/\varepsilon$. After this choice, take $n$ large enough for
Lemma~\ref{lem:aggregate} and \eqref{eq:rectangleenergy}.
Summing over the $O_B(n^2)$ pairs of levels and intervals retains a bound
$e^{-c_BN}$.

\subsection*{Large levels: $h>h_*$}
Apply Lemma~\ref{lem:bias} to $L$, using \eqref{eq:theta}. For a constant
$C_1$ depending only on the fixed parameters,
\begin{equation}\label{eq:largeB}
 |\cB|\le\frac{16Q}{\delta^2|L|}
 \le\frac{C_1e^{-h/2}}\rho\le\frac{2C_1e^{-h/2}}v,
 \qquad \ell_1=\lceil128\delta^{-2}\log(4Q)\rceil.
\end{equation}
The unconditional proportion of nonzero matching differences satisfying
$\langle V,A\rangle\in\cB$ is at least $g$.
Lemma~\ref{lem:matchingrank} therefore supplies $r$ independent differences
$V_1,\ldots,V_r$.

The forms $A\mapsto\langle V_j,A\rangle$ descend to arrays modulo
row-column additions, and remain independent in canonical coordinates.
Indeed, a matrix with zero row and column sums and zero upper-left
$(n-1)\times(n-1)$ block must be zero. Therefore restriction of balanced
matrices to that block is injective. The affine constant coordinate remains
free.

Choose an ordered independent tuple and targets $b_j\in\cB$. The constraints
$\langle V_j,A\rangle=b_j$, $1\le j\le r$, give an affine family of
dimension $d-r$, whose linear span has dimension at most $d-r+1$.
There are at most $T_k^r\le n^{2kr}$ choices of relation tuples.
The exponent from Lemma~\ref{lem:compression} is
$m_0-(d-r+1)+1=r-1$. Enumerating only descriptions obeying
\eqref{eq:largeB}, the contribution of this level and interval is at most
\begin{equation}\label{eq:largecount}
\begin{split}
 Q^{\ell_1}n^{2kr}\left(\frac{2C_1e^{-h/2}}v\right)^r
       2^{m_0}v^{r-1}
 =Q^{\ell_1}2^{m_0}(2C_1n^{2k})^re^{-hr/2}v^{-1}.
\end{split}
\end{equation}
Its logarithm is at most
\[
 O_B(N)+r(2k\log n-h/2)<O_B(N)-2r\log n.
\]
We used $\log(v^{-1})\le Bn$ and $\ell_1\log Q=O_B(N)$.
Since $r$ is a fixed positive proportion of $N$, this is
$e^{-\Omega_B(N\log n)}$. Summing over the $O_B(n^2)$ choices and
combining the two level ranges proves \eqref{eq:finite}.
\end{proof}

\begin{remark}[Order of parameter choices]
The parameters are chosen in the order $B$, then $k,\theta,\delta,g$, then
$\varepsilon$, and finally $n$. They are independent of the witness and its
Fourier level. The theorem does not assert uniformity for growing $B$.
The aggregate sampling lemma is used for one fixed positive $\varepsilon$
at a time.
\end{remark}

\section{Sharp asymptotics and the real remainder}\label{sec:real}
\subsection{The first-order asymptotic}
\begin{proof}[Proof of Theorem~\ref{thm:main}]
Fix $s\ge0$ and choose any fixed $B>1$. Bertrand's postulate gives a prime
satisfying \eqref{eq:prime}. Since $M_m$ is an integer matrix, real rank
deficiency implies rank deficiency modulo $Q$. Any nonzero modular normal
has either complexity below $t$, complexity at least $t$ and concentration
at least $e^{-Bn}$, or concentration below $e^{-Bn}$.
Proposition~\ref{prop:sparse}, Theorem~\ref{thm:finite}, and
Corollary~\ref{cor:rankcompression}, respectively, yield
\begin{equation}\label{eq:mainupper}
 \PP(\mathcal F_{n,s})\le\PP(\cE_m)+O(n^8e^{-2n})
       +e^{-c_Bn^2}+2\binom m{s+1}e^{-B(s+1)n}.
\end{equation}
Theorem~\ref{thm:finite} uses the first $d-1$ samples; the compression bound
uses all $m=d+s$ samples, with $D=d$ and $r=1$.

An empty position gives a nonzero real annihilating functional, so
$\cE_m\subseteq\mathcal F_{n,s}$. Two distinct positions are both avoided
by one sample with probability at most $1-2/n+1/[n(n-1)]$. The first two
Bonferroni inequalities therefore give, for fixed $s$,
\begin{equation}\label{eq:empty}
 \PP(\cE_m)=n^2(1-1/n)^m+O_s(n^4e^{-2n}).
\end{equation}
All errors in \eqref{eq:mainupper} are $o(n^2e^{-n})$, and
\[
 (d+s)\log(1-1/n)=-n+\frac32+O_s(n^{-1}).
\]
This proves \eqref{eq:main} for independent samples.

Let $\mathcal D_m$ be the event that the sampled permutations are distinct.
Then
\begin{equation}\label{eq:collision}
 \PP(\mathcal D_m^c)\le\frac{\binom m2}{n!}
       =\exp(-n\log n+O_s(n)).
\end{equation}
Conditioning on $\mathcal D_m$ and forgetting the order gives the uniform
$m$-element subset model. The conditioning error is smaller than the leading
term, proving the assertion for uniform subsets.
\end{proof}

\subsection{Lifting sparse modular relations}
To retain higher exponential orders, we work with the sparse event over $\R$.
The following rank comparison is used to pass between fields.

\begin{lemma}[Sparse lifting]\label{lem:lifting}
Let $S$ be a set of $k<n$ matrix positions, and let $Q$ be prime with
\begin{equation}\label{eq:hadamard}
 Q>(k+1)^{(k+1)/2}.
\end{equation}
If a nonzero affine assignment functional supported on $S$ over $\F_Q$
vanishes on a collection of permutation matrices, then a nonzero real affine
assignment functional supported on $S$ vanishes on that collection.
\end{lemma}
\begin{proof}
Write $S=\{e_1,\ldots,e_k\}$, and let $I_e(\pi)$ indicate whether
$\pi$ occupies position $e$. Form the matrix $B_S$ with rows
\[
 (1,I_{e_1}(\pi_r),\ldots,I_{e_k}(\pi_r)).
\]
Every $j\times j$ minor is an integer of absolute value at most
$j^{j/2}\le(k+1)^{(k+1)/2}$ by Hadamard's inequality.
Under \eqref{eq:hadamard}, its nonvanishing is preserved modulo $Q$.
Thus $\rank_\R B_S=\rank_{\F_Q}B_S$, and a nontrivial modular kernel
implies a nontrivial real kernel.

Every nonzero coefficient vector supported on $S$, including the constant,
defines a nonzero functional on the full permutation set. Otherwise its array
would have all rectangle differences zero. Since $k<n$, there is both a zero
row and a zero column. Anchoring there forces every entry to vanish, and then
the constant vanishes as well. Hence the real kernel vector gives a nonzero
assignment functional. It need not lift the particular modular vector.
\end{proof}

\begin{proposition}[Real concentrated normals outside the sparse event]
\label{prop:realconcentrated}
For every fixed $B>0$, there is $b_B>0$ such that, uniformly for $m\ge d-1$
and all sufficiently large $n$,
\begin{equation}\label{eq:realconcentrated}
 \PP\left(\begin{array}{c}
 \cL_{m,\R}^{\,c},\ \exists x\in\R^d\setminus\{0\}:\\
 M_mx=0,\ \rho_\R(x)\ge e^{-Bn}
 \end{array}\right)
 \le\exp\left(-\frac{b_Bn^2}{\log n}\right).
\end{equation}
\end{proposition}
\begin{proof}
First replace a real witness by a rational witness with at least the required
atom mass. Choose an atom of mass at least $e^{-Bn}$ and a permutation
$\pi_0$ in it. Impose the sample equations and the equations
\[
 x\cdot(Y_\pi-Y_{\pi_0})=0
 \quad\text{for every permutation $\pi$ in that atom}.
\]
This is a homogeneous system with integer coefficients. A nonzero real
solution implies a nonzero rational solution. Clear denominators and divide
by the greatest common divisor of the coordinates to obtain a primitive
integer vector $x$. Its reduction $\bar x$ modulo any prime is nonzero,
and reduction preserves all equalities in the chosen atom. Thus
\begin{equation}\label{eq:reduce}
 M_m\bar x=0,\qquad \rho_Q(\bar x)\ge e^{-Bn}.
\end{equation}

Choose $Q$ satisfying \eqref{eq:prime} and put
$K_n=\lfloor Bn/\log n\rfloor$. For large $n$, $K_n<n/4$, $K_n<t$,
and $(K_n+1)^{(K_n+1)/2}<Q$.
If $\kappa_Q(\bar x)\le K_n$, Lemma~\ref{lem:lifting}, applied to all
$m$ samples, gives a real relation of complexity below $t$, contradicting
$\cL_{m,\R}^{\,c}$.
If $K_n<\kappa_Q(\bar x)<t$, use \eqref{eq:mediumtail} and
\eqref{eq:entropytail} on the first $m_0=d-1$ samples. Their total probability
is at most
\[
 \sum_{K_n<k\le n/4}e^{-3kn/5}+e^{-cN}
 \le2e^{-3(K_n+1)n/5}+e^{-cN}.
\]
If $\kappa_Q(\bar x)\ge t$, Theorem~\ref{thm:finite}, again using the
first $m_0$ samples, bounds the probability by $e^{-c_BN}$.
Since $K_n+1>Bn/\log n$, these estimates imply
\eqref{eq:realconcentrated}.
\end{proof}

\begin{proof}[Proof of Theorem~\ref{thm:remainder}]
Fix $s\ge0$ and $C>0$, and choose a fixed $B>C$.
Outside the event in Proposition~\ref{prop:realconcentrated}, every real
normal annihilating the samples on $\cL_{m,\R}^{\,c}$ has concentration
less than $e^{-Bn}$. Apply Corollary~\ref{cor:rankcompression} over $\R$
with $D=d$ and $r=1$ to obtain
\begin{equation}\label{eq:realremainderbound}
 \PP(\mathcal F_{n,s}\setminus\cL_{m,\R})
 \le e^{-b_Bn^2/\log n}+2\binom m{s+1}e^{-B(s+1)n}=o(e^{-Cn}).
\end{equation}
For the uniform subset model, divide this upper bound by
$\PP(\mathcal D_m)\to1$. No additive collision error is needed here.
\end{proof}

\section{The second and third exponential orders}\label{sec:higher}
We compute the sparse event over fields of characteristic different from two,
then apply Theorem~\ref{thm:remainder} over $\R$. Two positions are
\emph{compatible} if their rows and columns are both different.
The occupancy column of a position $e$ is $(I_e(\pi_r))_{r=1}^m$.
Let $E_e$ be the event that the occupancy column at $e$ is zero.
For an unordered compatible pair $\{e,f\}$, let $B_{ef}$ be the event
that the columns at $e,f$ are equal and nonzero. For three distinct
positions with $\{e,f\}$ unordered, let $C_{g;ef}$ be the event that
$I_g(\pi_r)=I_e(\pi_r)+I_f(\pi_r)$ for every $r$, and both summand
columns are nonzero.
In the last definition, the equality is the ordinary equality of binary
integers. In particular the summand positions are never occupied together.
Let $\cU_m^{(3)}$ be the union of these elementary events.

\begin{lemma}[Relations on at most three positions]\label{lem:threepositions}
For every field $\mathbb K$ of characteristic different from two and all
sufficiently large $n$,
\begin{equation}\label{eq:threepositions}
 \cU_m^{(3)}\subseteq\cL_{m,\mathbb K},\qquad
 \PP(\cL_{m,\mathbb K}\setminus\cU_m^{(3)})=O(n^8e^{-4n}),
 \qquad m\ge d.
\end{equation}
The implicit constant is uniform over these fields and sample sizes.
\end{lemma}
\begin{proof}
Consider a homogeneous relation $aI_e+bI_f+cI_g=0$ with all three coefficients
nonzero. A pattern with exactly one occupied position is forbidden.
A two-position pattern is allowed exactly when its coefficients sum to zero.
If no pair sums to zero, only $000$ and possibly $111$ are allowed, so the
columns are empty or equal. If exactly one pair sums to zero, the remaining
position is always empty, because the total coefficient sum cannot vanish.
If two pairs sum to zero, the coefficients are proportional to $(1,1,-1)$
after reordering. Three vanishing pair sums are impossible in characteristic
different from two. In the case $(1,1,-1)$, if both allowed two-position
patterns occur, the event is $C_{g;ef}$; otherwise there is an empty column
or a pair equality. A nonzero pair equality necessarily involves compatible
positions. Relations on one or two positions give $E$ or $B$ in the same way.

For a nonzero affine constant and support of size at most three, vanishing
requires at least one support position to be occupied, an event of one-sample
probability at most $3/n$. For each support there are at most $2^{2^3}$ zero
sets on binary patterns. Thus all such exceptions have probability
$O(n^6(3/n)^m)$, which is smaller than the error in
\eqref{eq:threepositions}. Proposition~\ref{prop:sparse} handles supports
from $4$ to $t-1$.
Finally, each defining relation for $E,B,C$ is nonzero on the full permutation
set when $n>3$, by the zero-row and zero-column argument in
Lemma~\ref{lem:lifting}. Its complexity is at most three and hence less
than $t$ for large $n$. This proves the inclusion.
\end{proof}

\subsection{The exact two-position expression}
The numbers of unordered incompatible and compatible pairs are, respectively,
\[
 J=n^2(n-1),\qquad K=\frac{n^2(n-1)^2}{2}.
\]
The expression $T_{n,s}$ in \eqref{eq:T} sums the probabilities of single
empty events, subtracts their pair intersections, and adds all nonzero
pair-equality probabilities. Indeed, for a compatible pair and
$u=1/[n(n-1)]$,
\[
 \PP(E_e\cap E_f)=(1-2/n+u)^m,
\]
\[
 \PP(B_{ef})=(1-2/n+2u)^m-(1-2/n+u)^m.
\]
For an incompatible pair, simultaneous emptiness has probability
$(1-2/n)^m$ and nonzero equality is impossible. These contributions give
exactly \eqref{eq:T}, before accounting for intersections involving at least
three positions.

\subsection{The exact three-position correction}
Write $w=1/[n(n-1)(n-2)]$. Classify unordered triples of distinct positions
by the number $c$ of compatible pairs. Their counts are
\begin{equation}\label{eq:triplecounts}
\begin{split}
 N_0&=2n\binom n3,\qquad N_1=n^2(n-1)^2,\\
 N_2&=n^2(n-1)^2(n-2),\qquad N_3=6\binom n3^2.
\end{split}
\end{equation}
The four types are a row or column star, three corners of a rectangle,
a triple with exactly one incompatible pair, and a three-edge matching.
For $N_2$, choose the unique incompatible pair in $n^2(n-1)$ ways and
then the position compatible with both in $(n-1)(n-2)$ ways. For $N_3$,
choose three rows, three columns, and a bijection between them. The other
two counts are immediate, and their sum is $\binom N3$.

Define the signed correction terms
\begin{equation}\label{eq:Dcorrections}
\begin{split}
 D_0(m)={}&(1-3/n)^m,\\
 D_1(m)={}&2(1-3/n+u)^m-(1-3/n+2u)^m,\\
 D_2(m)={}&4(1-3/n+2u)^m-4(1-3/n+3u)^m\\
          &+(1-3/n+4u)^m,\\
 D_3(m)={}&9(1-3/n+3u-w)^m-9(1-3/n+4u-2w)^m\\
          &+3(1-3/n+5u-3w)^m-2(1-3/n+3u)^m,
\end{split}
\end{equation}
and put
\begin{equation}\label{eq:S}
 S_{n,s}=T_{n,s}+\sum_{c=0}^3N_cD_c(d+s).
\end{equation}
The quantities $D_c$ are corrections to a union formula, not probabilities
of individual events.

\begin{proposition}\label{prop:exactunion}
For $m=d+s\ge d$ in the independent sampling model,
\begin{equation}\label{eq:exactunion}
 \PP(\cU_m^{(3)})=S_{n,s}+O(n^8e^{-4n}).
\end{equation}
The error is uniform over $m\ge d$. The same formula holds for
$\PP(\cL_{m,\mathbb K})$, uniformly over fields of characteristic
different from two.
\end{proposition}
\begin{proof}
Fix a triple labelled $1,2,3$. Let $p_0$ be its one-sample probability of
pattern $000$, let $p_{ij}$ be the probability of exactly positions $i,j$
being occupied, and let $p_3$ be the probability of pattern $111$.
The correction to $T_{n,s}$ from the union of $E$ and $B$ events within
this triple is
\begin{equation}\label{eq:EBcorrection}
 p_0^m-\sum_{i<j}\bigl[(p_0+p_{ij})^m-p_0^m\bigr]
       -2\bigl[(p_0+p_3)^m-p_0^m\bigr].
\end{equation}
These terms account for the triple-empty intersection, an empty column
together with equality of the other two nonzero columns, and three identical
nonzero columns. In the last case, three pairwise intersections and one
triple intersection of equality events give coefficient $-3+1=-2$.

For $\{i,j,k\}=\{1,2,3\}$, the new event with $i$ as the sum coordinate
has probability
\begin{equation}\label{eq:Ccorrection}
 (p_0+p_{ij}+p_{ik})^m-(p_0+p_{ij})^m-(p_0+p_{ik})^m+p_0^m.
\end{equation}
Both pair patterns must occur. The three such events are mutually disjoint
and disjoint from the $E$ and $B$ events within the triple. Adding
\eqref{eq:EBcorrection} and the three terms in \eqref{eq:Ccorrection} gives
\begin{equation}\label{eq:triplecorrection}
 9p_0^m-3\sum_{i<j}(p_0+p_{ij})^m
       +\sum_{i=1}^3(p_0+p_{ij}+p_{ik})^m-2(p_0+p_3)^m,
\end{equation}
where in the second sum $\{j,k\}=\{1,2,3\}\setminus\{i\}$.
For $c<3$, one has $p_0=1-3/n+cu$, $p_3=0$, and $p_{ij}=u$ for
compatible pairs and zero otherwise. For $c=3$, one has
$p_0=1-3/n+3u-w$, $p_{ij}=u-w$, and $p_3=w$.
Substitution into \eqref{eq:triplecorrection} gives
\eqref{eq:Dcorrections}.

It remains to bound intersections not confined to at most three positions.
Any intersection of elementary events whose supports have union of size $K$
forbids a sample from occupying exactly one position of that union: the
position has a nonzero coefficient in at least one defining equation.
The one-hit estimate used in \eqref{eq:onehit} gives a forbidden probability
of at least
\[
 \frac Kn-\frac{K(K-1)}{n(n-1)}.
\]
For fixed $K$, the intersection over $m\ge d$ samples therefore has
probability $O_K(e^{-Kn})$.

Apply inclusion-exclusion through intersections of three elementary events,
with the fourth-order Bonferroni remainder. At most three distinct elementary
events on a fixed set of three positions can occur simultaneously. Indeed,
the possibilities are empty columns, groups of equal nonzero columns, or
one $C$ event, which excludes the other events on that triple.
Thus every nonempty intersection of four distinct events involves at least
four positions. All discarded intersections, including those of two or three
events whose support union has size at least four, involve $4\le K\le12$
positions. For fixed $K$, there are $O_K(n^{2K})$ choices, since each
intersection has at most four events and each event uses at most three
positions. Their total contribution is bounded by
\[
 O\left(\sum_{K=4}^{12}n^{2K}e^{-Kn}\right)=O(n^8e^{-4n}).
\]
This proves \eqref{eq:exactunion}; Lemma~\ref{lem:threepositions} proves
the sparse-normal version.
\end{proof}

\subsection{Asymptotic coefficients}
\begin{proof}[Proof of Theorem~\ref{thm:higher}]
Since $\cL_{m,\R}\subseteq\mathcal F_{n,s}$,
Proposition~\ref{prop:exactunion} and Theorem~\ref{thm:remainder} with
$C=4$ give \eqref{eq:exacthigher}.
For fixed $a,b,s$,
\begin{equation}\label{eq:expansion}
 \left(1-\frac an+\frac b{n^2}+O(n^{-3})\right)^{d+s}
 =e^{-an+2a+b-a^2/2}\bigl(1+O_s(n^{-1})\bigr),
\end{equation}
where the implicit constant may also depend on $a,b$ and the remainder.
In \eqref{eq:T}, the incompatible-pair contribution is $O(n^3e^{-2n})$,
whereas $K=(1/2+O(n^{-1}))n^4$. Formula~\eqref{eq:expansion} gives the
coefficient $(e^4-2e^3)/2$. All triple corrections are
$O_s(n^6e^{-3n})$, proving \eqref{eq:second}.

For the third order, $N_c=O(n^5)$ when $c<3$, whereas
$N_3=(1/6+O(n^{-1}))n^6$. Since $w=O(n^{-3})$, the terms involving $w$
do not change the leading constants in \eqref{eq:expansion}. Thus
\[
 D_3(d+s)=\bigl[e^{9/2}(9-9e+3e^2-2)+O_s(n^{-1})\bigr]e^{-3n}.
\]
Together with \eqref{eq:S}, this proves \eqref{eq:third}.
Finally, the conditioning error in \eqref{eq:collision} is
$o(n^8e^{-4n})$, so all three formulas also hold for uniform subsets.
\end{proof}

\begin{remark}[Characteristic two]
The characteristic restriction in Lemma~\ref{lem:threepositions} is necessary.
Over $\F_2$, the relation $I_1+I_2+I_3=0$ permits all three pair patterns
$110,101,011$. If all three occur, the corresponding columns are independent
over $\R$: the $3\times3$ matrix formed by these rows has determinant $-2$.
Working with real sparse relations and odd primes avoids this discrepancy.
\end{remark}

\Needspace{14\baselineskip}
\section*{Disclosure}
AI tools were used substantially during the preparation of the first version
of this paper, as described in the disclosure accompanying that version. The
principal mathematical advances in the present version---including the sharp
asymptotic, the higher exponential-order expansions, and the new finite-field,
permanental-minor, counting, and lifting arguments used to prove them---were
subsequently developed by the author. In preparing the present version, AI
tools were also used for editorial assistance, including revising wording and
presentation. The author determined the mathematical strategy and proofs of
the new results, verified the mathematical content, and takes full
responsibility for all mathematical content and exposition.

\end{document}